\documentclass[11pt]{article}
\usepackage{amssymb}
\usepackage{latexsym}
\usepackage{amsthm}
\usepackage{amscd}
\usepackage{amsmath}
\usepackage{diagrams}

\usepackage{indentfirst}
\theoremstyle{definition}
\newtheorem{thm}{Theorem}[section]
\newtheorem{lem}[thm]{Lemma}
\newtheorem{th-def}[thm]{Theorem-Definition}
\newtheorem{cor}[thm]{Corollary}
\newtheorem{defn-lem}[thm]{Definition-Lemma}

\newtheorem{prop}[thm]{Proposition}

\newtheorem{rem}[thm]{Remark}

\numberwithin{equation}{section}

\allowdisplaybreaks[4]

\def \Q{{\mathbb Q}}

\def \C{{\mathbb C}}
\def \F{{\mathbb F}}

\def \Z{{\mathbb Z}}

\def\map#1.#2.{#1 \longrightarrow #2}
\def\rmap#1.#2.{#1 \dasharrow #2}

\DeclareMathOperator{\Pic}{Pic}

\DeclareMathOperator{\Spec}{Spec}
\DeclareMathOperator{\Spf}{Spf}

\def\fb#1.{\underset #1 \to \times}
\def\pr#1.{\Bbb P^{#1}}
\def\ring#1.{\mathcal O_{#1}}
\def\mlist#1.#2.{{#1}_1,{#1}_2,\dots,{#1}_{#2}}

\def\uloopr#1{\ar@'{@+{[0,0]+(-4,5)} @+{[0,0]+(0,10)}
@+{[0,0]+(4,5)}}
  ^{#1}}
\def\dloopr#1{\ar@'{@+{[0,0]+(-4,-5)} @+{[0,0]+(0,-10)}
@+{[0,0]+(4,-5)}}
  _{#1}}

\def\rloopd#1{\ar@'{@+{[0,0]+(5,4)} @+{[0,0]+(10,0)}
@+{[0,0]+(5,-4)}}
  ^{#1}}
\def\lloopd#1{\ar@'{@+{[0,0]+(-5,4)} @+{[0,0]+(-10,0)}
@+{[0,0]+(-5,-4)}}
  _{#1}}

\long\def\ignore#1{}
\long\def\ignore#1{#1}
\title{N\'{e}ron-Severi group preserving lifting of K3 surfaces and applications}
\author{Junmyeong Jang}
\date{}
\begin{document}

\maketitle
\vspace{0.3cm}
\begin{center}
{\Large
Abstract}
\end{center}
\small
For complex algebraic K3 surfaces, there are  criteria in terms of N\'{e}ron-Severi lattice to be a Kummer surface and to be an Enriques K3 surface. In this paper, using a N\'{e}ron-Severi group preserving lifting, we prove the same criteria hold for K3 surfaces of finite height over a field of odd characteristic. We also give a classification of K3 surfaces of Picard number 20 over a field of odd characteristic.

\vspace{0.3cm}
%\begin{center}
%Mathematics Subject Classification : 11G25, 14J20
%\end{center}
%\begin{center}
%{Junmyeong Jang}
%\end{center}
\normalsize

\medskip

      \section{Introduction}
      When $X$ is a complex algebraic K3 surface, the second integral singular cohomology group $H^{2}(X, \Z)$ is an integral lattice which is isomorphic to $U^{3} \oplus (-E_{8}) ^{2}$. Here $U$ is a hyperbolic unimodular lattice of rank 2 and $E_{8}$ is a  unimodular root lattice
      of rank 8. By Hodge theory, we may regard the global 2-forms of $X$, $H^{0}(X,\Omega _{X/\C}^{2})$ as a direct factor of $H^{2}(X,\C) = H^{2}(X, \Z) \otimes \C$. The N\'{e}ron-Severi group of $X$, $NS(X)$ is identified with the intersection of the orthogonal complement of $H^{0}(X,\Omega _{X/\C}^{2})$ and $H^{2}(X,\Z)$ inside $H^{2}(X,\C)$. We call the rank of $NS(X)$ the Picard number of $X$.
       Thanks to the Torelli theorem for complex K3 surfaces (\cite{PS}), we can express many geometric properties of complex K3 surfaces in terms of N\'{e}ron-Severi lattice.\\

      Let $Y$ be a complex Enriques surface.
      The free part of the N\'{e}ron-Severi group $Y$ admits a lattice structure
      $\Gamma = U \oplus -E_{8}$.
      We say $\Gamma$ is the Enriques lattice.
       We denote by $\Gamma (2)$ the lattice obtained from $\Gamma$ by multiplying the pairing by 2.
      Let $f: X \to Y$ be the K3 cover of $Y$ and $\iota : X \to X$ be the associated involution.
       We say a K3 surface is an Enriques K3 surface if it is the K3 cover of an Enriques surface.
       A K3 surface is an Enriques K3 surface if and only if it has a fixed point free involution.
            The covering map induces an embedding of lattices
      $$f^{*} : \Gamma (2) \hookrightarrow NS(X).$$
      Then $f^{*} \Gamma (2)$ is identified with the fixed part of $NS(X)$ of the involution (\cite{Na}, Proposition 2.3)
      $$ \iota ^{*} :NS(X) \to NS(X).$$
      Let $\Gamma (2) ^{\bot}$ be the orthogonal complement of $f ^{*} : \Gamma (2) \hookrightarrow NS(X)$.
      Then
      $$\Gamma (2) ^{\bot} = \{ x \in NS(X) | \iota ^{*}(x) = -x \}$$
      and $\Gamma (2) ^{\bot}$ is an even negative definite lattice.
       Assume the self intersection of a vector $x$ in $\Gamma (2) ^{\bot}$ is $-2$. By the Riemann-Roch theorem, either of $x$ or $-x$ is effective, but not both.
        However $x$ is effective if and only if $\iota ^{*}(x) = -x$ is effective. Therefore $\Gamma (2) ^{\bot}$ does not contain a vector of self intersection $-2$.
      Based on the Torelli theorem, the converse is also valid and we have a lattice criterion for a complex K3 surface to be an Enriques K3 surface.
      \\

      \vspace{0.2cm}
      \textbf{Theorem.}(\cite{Na}, Theorem 1.14, \cite{Ke}, Theorem 1)
      A complex K3 surface $X$ is an Enriques K3 surface if and only if there exists a primitive embedding of lattices
      $\Gamma(2) \hookrightarrow NS(X)$ such that the orthogonal complement of the embedding does not contain a vector of self intersection $-2$.\\

      \vspace{0.2cm}
      Now assume $X$ is a complex Kummer surface associated to an abelian surface $A$. Then $X$ is the blow up of 16 singular points of $A_{1}$-type of $A/(-id_{A})$.
      Let $I = \{ a_{1} , \cdots , a_{16} \}$ be the set of 2 torsion points of $A$. By the group structure of $A$, we may regard $I$ as a 4-dimensional vector space over $\F_{2}$, where $\F_{2}$ is a prime field of characteristic 2.
      Inside $I$,
      there are 30 hyperplanes and each hyperplane has 8 points. Let $Q$ be the subset of the power set of $I$ consisting of 30 hyperplanes, the empty set and $I$. For each $a_{i} \in I$, there is a smooth rational curve on $X$ corresponding to $a_{i}$. We denote this smooth rational curve by $a_{i}$ again.
      The classes of 16 rational curves $a_{i}$ are linearly independent in $NS(X)$. Let $M$ be the sublattice of $NS(X)$ generated by the 16 classes of $a_{i}$.
            For each $\alpha \in Q$, let $v_{\alpha} = \frac{1}{2}\sum _{a_{i} \in \alpha} a_{i} \in M \otimes \Q$. We set
      $$J = M + \sum_{\alpha \in Q} \Z \cdot v_{\alpha} \subset M \otimes \Q.$$
      Then $J$ is the saturation of $M$ in $NS(X)$ and $J/M = (\Z/2)^{5}$. We call $J$ the Kummer lattice.
      The following lattice criterion is also based on the Torelli theorem.
      \\

      \vspace{0.2cm}
      \textbf{Theorem.}(\cite{Ni}, Theorem 3)
      A complex K3 surface $X$ is a Kummer surface if and only if there exists a primitive embedding of the abstract lattice $J$ into $NS(X)$.\\

      \vspace{0.2cm}

      It is also known that every complex Kummer surface is an Enriques K3 surface (\cite{Ke}, Theorem 2).\\

      The Torelli theorem for complex K3 surfaces states that the isomorphism classes of algebraic complex K3 surfaces can be classified
      by the location of the line of global holomorphic 2 forms in the second complex singular cohomology. Especially for complex K3 surfaces of Picard number 20, we have a more precise classification in terms of transcendental lattice. The transcendental lattice of a complex K3 surface $X$ is the orthogonal complement of the embedding
      $$NS(X) \hookrightarrow H^{2}(X,\Z).$$
      We denote the transcendental lattice of $X$ by $T(X)$.
      For a complex K3 surface of Picard rank 20, $T(X)$ is an even positive definite lattice of rank 2. \\

      \vspace{0.2cm}
      \textbf{Theorem.}(\cite{IS}, Theorem 4)
      The correspondence
      $$X \mapsto T(X)$$
      gives a bijection between the set of isomorphism classes of complex K3 surfaces of Picard number 20 and the set of isomorphism classes of even positive definite lattices of rank 2. Every complex K3 surface of Picard number 20 has a model over a number field.
      \\

      \vspace{0.2cm}

      Over a field of positive characteristic, we do not have the Torelli theorem for K3 surfaces.
      Let $k$ be an algebraically closed field of positive characteristic $p$. Assume $X$ is a K3 surface over $k$.
      The formal Brauer group of $X$, $\widehat{Br}_{X}$ is a 1-dimensional smooth formal group over $k$.
      The height of $\widehat{Br}_{X}$ is an integer between 1 and 10 or infinite. We say the height of $X$ is the height of $\widehat{Br}_{X}$.\\

      A K3 surface of infinite height is also called a supersingular K3 surface.
      It is known that if the base characteristic is odd, the Picard number of
       a supersingular K3 surface is 22 (\cite{C}, \cite{MA}).
       It is also known that if the base characteristic is at least 5, a supersingular K3 surface is unirational (\cite{Li}).
       The discriminant group of the N\'{e}ron-Severi group of a supersingular K3 surface $X$, $(NS(X)^{*})/NS(X)$ is a $\sigma$-dimensional
       space over a prime field $\Z / p$ for $1 \leq \sigma \leq 10$.
       We call $\sigma$ the Artin invariant of $X$.
      For supersingular K3 surfaces over a field of odd characteristic, we have the crystalline Torelli theorem (\cite{Og2}).
      In a previous work (\cite{Ja}, Theorem 4.1), using the crystalline Torelli theorem, for a supersingular K3 surface in odd characteristic $p$,
      we proved the lattice criterion to be an Enriques K3 surface holds.
       By this result and some lattice
      calculation, we also proved that when $p \geq 23$, a supersingular K3 surface is an Enriques K3 surface if and only if the Artin invariant is less than $6$. It is known that a supersingular K3 surface over odd characteristic is a Kummer surface if and only if the Artin invariant is less than 3 (\cite{Og1}, Theorem 7.10).
      Therefore if $p \geq 23$, every supersingular Kummer surface is an Enriques K3 surface.
      In this paper, we use a deformation argument to prove that over any odd characteristic, a supersingular K3 surface is an Enriques K3 surface if and only if the Artin invariant is less than 6 (Corollary 2.4).\\

      For a K3 surface of finite height, we do not have a proper replacement of the Torelli theorem.
      However over a field of odd characteristic, any K3 surface of finite height has a smooth lifting over the ring of Witt vectors of the base field to which
      all the line bundles  can extend.\\

      \vspace{0.2cm}
      \textbf{Theorem.}(\cite{NO}, p.505,  \cite{LM}, Corollary 4.2)
      Let $k$ be an algebraically closed field of odd characteristic and $W$ be the ring of Witt vectors of $k$. Assume $X$ is a K3 surface of finite height defined over $k$.
      Then there exists a smooth lifting of $X$ over $W$, $\mathbb{X}/W$ such that the reduction map
      $\Pic (\mathbb{X}) \to NS(X)$ is surjective.
      \\

      \vspace{0.2cm}

      For such a lifting, the reduction map from the N\'{e}ron-Severi group of the generic fiber
      to the N\'{e}ron-Severi group of the special fiber, $X$ is an isomorphism.
      We say a lifting of $X$ satisfying this condition is a N\'{e}ron-Severi group preserving lifting of $X$.
      For a K3 surface of finite height, using a N\'{e}ron-Severi group preserving lifting, we may import
      some results on complex K3 surfaces expressed in terms of N\'{e}ron-Severi groups.\\

        In this paper, we use a N\'{e}ron-Severi group preserving lifting to prove the lattice criteria to be an Enriques K3 surface and to be a Kummer surface still hold
        for a K3 surface of finite height in odd characteristic. Also we prove a Kummer surface of finite height is an Enriques K3 surface.
        Therefore all these results hold for all the K3 surfaces over any filed of characteristic $p \neq 2$.

      \vspace{0.2cm}
      \textbf{Theorem 2.5.}
      Assume $k$ is an algebraically closed field of characteristic $p > 2$. A K3 surface $X$ over $k$ is an Enriques K3 surface
      if and only if there exists a primitive embedding
     $\Gamma(2) \hookrightarrow NS(X)$ such that the orthogonal complement of the embedding does not have a vector of self intersection $-2$.
     \\

     \vspace{0.2cm}

     \textbf{Theorem 2.6.}
     Assume $k$ is an algebraically closed field of characteristic $p > 2$. A K3 surface $X$ over $k$ is a Kummer surface if and only if
     there exists a primitive embedding of $J$ into $NS(X)$.\\

     \vspace{0.2cm}
      \textbf{Theorem 2.8.}
      Assume $k$ is an algebraically closed field of characteristic $p > 2$. A Kummer surface $X$ over $k$ is an Enriques K3 surface.\\

      \vspace{0.2cm}
      Over a field of odd characteristic, the Picard number of a K3 surface of finite height $h$ $(1 \leq h \leq 10)$
      is at most $22-2h$  (\cite{I1}, Proposition 5.12). Therefore
      if the Picard number of a K3 surface $X$ is 20, the height of $X$ is 1, in other words, $X$ is ordinary.
       Moreover in this case, there is a unique N\'{e}ron-Severi group preserving lifting of $X$, which is the canonical lifting of $X$ (\cite{N2}, Definition 1.9). By this fact and the classification of complex K3 surfaces of Picard number 20,
      we obtain a classification of K3 surfaces of Picard number 20 over a field of odd characteristic $p$.
      Let $S_{p}$ be the set of isomorphic classes of even positive definite lattices of rank 2 such that the discriminant is a non-zero square modulo $p$. For each $M \in S_{p}$, there is a unique complex K3 surface of Picard number 20, $X_{M}$ such that $T(X_{M})$ is isomorphic to $M$.
      $X_{M}$ is defined over $\bar{\Q}$.
      When $k$ is an algebraically closed field of characteristic $p$,
       $X_{M}$  has a good reduction over $k$. The reduction of $X_{M}$ over $k$ is a K3 surface of Picard number 20. In section 3, using N\'{e}ron-Severi group preserving lifting,
       we prove that this construction gives a classification of K3 surfaces of Picard number 20 over $k$.
      \\
      \vspace{0.2cm}

      \textbf{Theorem 3.7.}
      Let $k$ be an algebraically closed field of characteristic $p>2$.
    The isomorphism classes of K3 surfaces of Picard number 20 are classified by $S_{p}$.
      Every K3 surface of Picard number 20 over $k$ has a model over a finite field.\\

      \vspace{0.6cm}
    {\bf Acknowledgment}\\
        The author thanks to C.Liedtke for helpful comments.
        He is also grateful to the referee for pointing out mistakes in the early version of this paper.
        This research was supported by Basic Science Research Program through the National Research Foundation of Korea(NRF) funded by the Ministry of Education, Science and Technology(2011-0011428).

     \section{Kummer surfaces and Enriques surfaces}
      Let $k$ be an algebraically closed field of characteristic $p>2$. Let $W$ be the ring of Witt vectors of $k$ and $K$ be the fraction field of $W$.
     Assume $Y$ is an Enriques surface over $k$ and $f:X \to Y$ is the K3 cover of $Y$.
     Since $f$ is \'{e}tale,
     \begin{center}
     $f^{*} T_{Y/k} = T_{X/k} = \Omega ^{1}_{X/k}$ and $f_{*} \Omega _{X/k} ^{1} = \Omega _{Y/k}^{1} \oplus T_{Y/k}$,
     \end{center}
     so
     \begin{center}
     $H^{0}(Y, T_{Y/k}) = H^{2}(T,T_{Y/k})=0$ and
     $\dim H^{1}(Y, T_{Y/k})=10$.
     \end{center}
     It follows that the deformation space of $Y$ to $k$-Artin local algebras
     with residue field $k$
     is formally smooth of 10 dimension over $k$. Let $\mathcal{S} = \Spf k[[t_{1} , \cdots , t_{10}]] $ be the deformation space. Suppose $(A,m)$ is an Artin local $k$-algebra with residue field $A/m \simeq k$ and $Y_{A} \in \mathcal{S}(A)$ is a deformation of $Y$ over $A$.
     Let $i : Y \hookrightarrow Y_{A}$ be the canonical embedding.
     \begin{lem}\label{lem1}
     The canonical map $i ^{*} : \Pic(Y_{A}) \to \Pic (Y)$ is an isomorphism.
     \end{lem}
     \begin{proof}
    Let us consider the exact sequence
     $$ 0 \to 1+m\mathcal{O}_{Y_{A}} \to \mathbb{G} _{m, Y_{A}} \to \mathbb{G} _{m,Y} \to 0.$$
     Since $A$ is a finite algebra over $k$, $H^{2} (Y, 1+m\mathcal{O}_{Y_{A}})$ is an iterative extension of finite copies of $H^{2}(Y, \mathcal{O}_{Y})$,
     so $H^{2}(Y , 1+ m\mathcal{O}_{Y})=0$. By the same reason, $H^{1}(Y, 1+ m\mathcal{O}_{Y}) =0$.
     Hence $H^{1}(Y_{A}, \mathbb{G} _{m,Y_{A}}) \to H^{1}(Y, \mathbb{G} _{m,Y})$ is an isomorphism.
     \end{proof}
     Let $\mathcal{Y} \to \mathcal{S}$ be the universal family over the deformation space $\mathcal{S}$. By Lemma \ref{lem1},
     all the line bundles of $Y$ extend to $\mathcal{Y}$. In particular, an ample line bundle extends to $\mathcal{Y}$ and $\mathcal{Y} \to
     \mathcal{S}$ is algebraizable. Let $S = \Spec k[[t_{1}, \cdots , t_{10}]]$ and $Y_{S} \to S$ be the algebraic model of the formal
     scheme $\mathcal{Y} \to \mathcal{S}$. If $f _{S} : X_{S} \to Y_{S}$ is the K3 cover of $Y_{S}$, $\pi : X_{S} \to S$ is a family of
     Enriques K3 surfaces over $S$.
     \begin{lem}[c.f. \cite{MN}, p.383]\label{lem2}
     If $X$ is a supersingular K3 surface of Artin invariant 1 over $k$, $X$ is an  Enriques K3 surface.
          \end{lem}
     \begin{proof}
     Let $E$ be a supersingular elliptic curve over $k$. Then $X$ is isomorphic to the Kummer surface of the abelian surface $E \times E$ (\cite{Og1}, Corollary 7.14).
     Assume $a$ is a non-zero 2-torsion point of $E$. An involution
     $$ E \times E \to E \times E,\ (x,y) \mapsto (-x+a, y+a)$$
     commutes with $-id _{E \times E}$ and induces a fixed point free involution on $X$. Hence $X$ is an Enriques K3 surface.
     \end{proof}
     Assume $Y$ is an Enriques surface whose K3 cover $X$ is a supersingular K3 surface of Artin invariant 1. For the family of Enriques K3 surfaces
     $\pi : X_{S} \to S$, let us consider a stratification on $S$,
\begin{equation}\label{fil}
     S=M_{1} \supset M_{2} \supset \cdots \supset M_{10} \supset M_{11} = \Sigma _{10} \supset \Sigma _{9} \supset \cdots \supset \Sigma _{1}.
\end{equation}
     Here $M_{i}$ is the locus of fibers of height at least $i$ and $\Sigma _{i}$ is the locus of supersingular fibers of Artin invariant
     at most $i$. All $M_{i}$ and $\Sigma _{i}$ are closed subset of $S$ and, for the simplicity, we assume they are all reduced.
     Each step in the stratification is defined by one equation,
     so the dimension decreases by at most 1 for each step (\cite{A},\textrm{ p.563}).
     Since the central fiber of
     $\pi$ is $X$, the closed point of $S$ is contained in $\Sigma _{1}$.
     On the other hands,
     a K3 surface of Artin invariant 1 is unique up to isomorphism. Hence $\Sigma _{1}$ is zero-dimensional and consists of one point.
     It follows that there are exactly
     10 steps of dimension down on (\ref{fil}). If an Enriques K3 surface is of finite height, the height is at most 6 and if an Enriques K3 surface is supersingular, the Artin invariant is at most 5 (\cite{Ja}, Proposition 3.5).
     Therefore $M_{6} =M_{10}$ and $\Sigma _{10} = \Sigma _{5}$, and
     if we shorten the stratification (\ref{fil}) as
     $$S= M_{1} \supset M_{2} \supset \cdots \supset M_{6} \supset \Sigma _{5} \supset \Sigma _{4} \supset \cdots \supset \Sigma _{1},$$
     each step is of codimension 1.
     \begin{thm}\label{thm1}
           For each $h$ $(1 \leq h \leq 6)$, there exists an Enriques K3 surface of height $h$ over $k$. For each $\sigma$ $(1 \leq \sigma \leq 5)$, there exists a supersingular Enriques K3 surface of Artin invariant $\sigma$ over $k$.
     \end{thm}
     \begin{proof}
     We choose a point $x \in M_{h} - M_{h+1}$ or $x \in M_{6} - \Sigma _{5}$ when $h=6$. The fiber over $x$,
     $X_{x} = X_{S} \times _{S} k(x)$ is an Enriques K3 surface of height $h$ over $k(x)$.
     Let $X_{B} \to \Spec B$ be an integral model of $X_{x} \to \Spec k(x)$ with the Enriques involution where
     $B$ is an integral domain of finite type over $k$ with an imbedding
     $B \hookrightarrow k(x)$ such that $X_{B} \otimes k(x)$ is isomorphic to $X_{x}$. We may assume every fiber of $X_{B} \to \Spec B$
     is an Enriques K3 surface of height $h$
     considering a stratification similar to (\ref{fil}) on $\Spec B$.
     Hence a closed fiber of $X_{B} \to \Spec B$ is an Enriques K3 surface of height $h$ over $k$.
     In a similar way, we can construct a supersingular Enriques K3 surface of Artin invariant $\sigma$ for $1\leq \sigma \leq5$.
     \end{proof}
In a previous work (\cite{Ja}, Corollary 4.7), we proved that if the characteristic of the base field is $p \geq 23$, a supersingular K3 surface has an Enriques
involution if and only if the Artin invariant is less than 6.
It is also known that if one supersingular K3 surface of Artin invariant $\sigma$ over $k$ is an Enriques K3 surface, then every supersingular K3 surface of Artin invariant $\sigma$ over $k$ is supersingular (\cite{Ja}, Remark 4.4). By Theorem 2.3, we obtain the following result in any odd characteristic.

\begin{cor}\label{sup}
     Let $k$ be an algebraically closed field of odd characteristic.
     A supersingular K3 surface over $k$ is an Enriques K3 surface if and only if the Artin invariant is less than 6.
     \end{cor}
     Now we prove that, for a K3 surface of finite height in odd characteristic, the lattice criteria to be an Enriques K3 surface and the lattice criterion to be a Kummer surface hold
               We recall from Section 1 that $\Gamma$ is the Enriques lattice and $J$ is the Kummer lattice.
           \begin{thm}\label{Enr}
     Assume $k$ is an algebraically closed field of characteristic $p > 2$. A K3 surface $X$ over $k$ is an Enriques K3 surface
      if and only if there exists a primitive embedding
     $\Gamma(2) \hookrightarrow NS(X)$ such that the orthogonal complement of the embedding does not have a vector of self intersection $-2$.
     \end{thm}
     \begin{proof}
     The only if part comes from the Riemann-Roch theorem.
     For the supersingular case, we refer to \cite{Ja}, Theorem 4.1.
     Assume $X$ is of finite height and there is an embedding $\Gamma (2) \hookrightarrow NS(X)$ satisfying the given condition. We fix a N\'{e}ron-Severi group preserving lifting of $X$
     over $W$, $\pi : \mathbb{X} \to \Spec W$. Let $X_{K}$ be the generic fiber of $\mathbb{X}/W$.
     Since $NS(X) = NS(X _{K})$ and $K$ is a field of characteristic 0, $X_{K} \otimes \bar{K}$ has
     an Enriques involution $g$. We assume $g$ is defined over a finite extension $L$ of $K$ and $V$ is the ring of integers of
     $L$. Then $g$ extends to $\mathbb{X} \otimes V$ and $\bar{g} = g|X$ is an involution on $X$ (\cite{MM}, p.672, Corollary 1).
Because $g^{*}|H^{0}(X_{K},
     \Omega _{X_{K}/K} ^{2}) = -1$, $g^{*}|\pi _{*} \Omega ^{2}_{\mathbb{X} \otimes V /V}=-1$ and $\bar{g} ^{*}|H^{0}(X, \Omega _{X/k} ^{2})=-1$. Also $\bar{g}^{*}$ fixes a rank 10 sublattice of $NS(X)$. Hence $\bar{g}$ is an Enriques
     involution of $X$.
     \end{proof}

     \begin{thm}
     Assume $k$ is an algebraically closed field of characteristic $p > 2$. A K3 surface $X$ over $k$ is a Kummer surface if and only if
     there exists a primitive embedding of $J$ into $NS(X)$.
     \end{thm}
     \begin{proof}
     For the only if part, we refer to \cite{Ni}, Section 1.
     When $X$ is supersingular, $X$ is a Kummer surface if and only if the Artin invariant of $X$ is 1 or 2
     (\cite{Og1}, Theorem 7.10). Assume the Artin invariant of $X$ is greater than 2. Since $J \otimes \Z_{p}$ is a unimodular $\Z _{p}$-lattice of rank
     16 with a square discriminant, by \cite{Sh2}, Theorem 1.1, there is no embedding of $J \otimes \Z _{p}$ into $NS(X) \otimes \Z _{p}$.
     Therefore if there is an embedding of $J$ into $NS(X)$, $X$ is a supersingular Kummer surface.
     Suppose $X$ is of finite height and there is a primitive embedding of $J$ into $NS(X)$. Let $\mathbb{X}$ be a N\'{e}ron-Severi group preserving lifting of $X$ over $W$ and $X_{K}$ be the generic fiber of $\mathbb{X}/ W$. Since there is a primitive embedding
     $$J \hookrightarrow NS(X) = NS(X_{K}),$$
      $X_{K}$ is a Kummer surface. Therefore $X_{K}$ contains 16 mutually disjoint smooth rational curves $C_{1}, \cdots , C_{16}$ satisfying $\frac{1}{2} \sum C_{i} \in NS(X_{K})$.

     \begin{lem}[c.f. \cite{LM}, Lemma 2.3]
     A class $ v \in NS(X)$ is effective on $X$ if and only if it is effective on $X_{K}$.
     A class $v \in NS(X)$ represents a smooth rational curve if and only if $v \in NS(X_{K})$ represents a smooth rational curve.
          \end{lem}
     \begin{proof}
     Let us fix a class $h \in NS(X)$
     which is ample on both of $X$ and $X_{K}$.
     Assume $v \in NS(X)$ is an effective class on $X$ and $D$ is a curve represented by $v$.
     We assume $D = \sum n_{i} C_{i}$ for integral curves $C_{i}$. Let $w_{i}$ be the class in $NS(X)$ which
     represents $C_{i}$. Then $(h,w_{i}) >0$ and $(w_{i}, w_{i}) \geq -2$ by the adjunction formula.
     By the Riemann-Roch theorem, $w_{i}$ is an effective class in $NS(X_{K})$, so $v= \sum n_{i}w_{i}$ is also effective in $NS(K_{X})$.
     The converse follows by the same way. This also asserts that $ v \in NS(X)$ is indecomposable effective if and only if $v \in NS(X _{K})$ is
     indecomposable effective. A class $v \in NS(X)$ represents a smooth rational curve if and only if $(v,v)=-2$ and $v$ is indecomposable effective.
     Then $v \in NS(X_{K})$ also represents a smooth rational curve. The converse follows by the same way.
     \end{proof}

     By Lemma 2.7, the reduction of each $C_{i}$ is a smooth rational curve in $X$. We denote the reduction of $C_{i}$ by $\bar{C} _{i}$.
     Let $X' \to X$ be the double cover ramified along the 16 rational curves $\bar{C}_{i}$. The preimage of each $\bar{C}_{i}$ in $X'$ is a ($-1$)-curve. Let $A$ be the surface obtained by blowing down the 16 ($-1$)-curves of $X'$.
     It can be checked that $A$ is an abelian surface and $X$ is the Kummer surface of $A$.
     \end{proof}

     \begin{thm}
     Assume $k$ is an algebraically closed field of characteristic $p > 2$. A Kummer surface $X$ over $k$ is an Enriques K3 surface.
     \end{thm}
     \begin{proof}
     If $X$ is a supersingular Kummer surface, the Artin invariant of $X$ is at most 2. Then $X$ is an Enriques K3 surface by Corollary \ref{sup}.
     Assume $X$ is a Kummer surface of finite height.
     Let $\pi : \mathbb{X} \to \Spec W$ be a N\'{e}ron-Severi group preserving lifting of $X$ and $X_{K}$ be the generic fiber of $\pi$.
     Since $X$ is a Kummer surface, there exists a primitive embedding
     $$J \hookrightarrow NS(X) = NS(X _{K})$$
     and
     $X_{K}$ is a Kummer surface. Since $X_{K}$ is defined over a field of characteristic 0, $X_{K}$ is an Enriques K3 surface and there exists
     a primitive embedding
     $$\Gamma \hookrightarrow NS(X_{K}) = NS(X)$$
      such that the orthogonal complement does not contain a vector of
     self intersection $-2$. By theorem \ref{Enr}, $X$ is an Enriques K3 surface.
     \end{proof}

     \section{Classification of K3 surfaces of Picard number 20}
     A singular K3 surface is a complex K3 surface of Picard number 20. For a complex K3 surface $X$,
     we denote the transcendental lattice of $X$ by $T(X)$. The transcendental lattice of a singular K3 surface is an even
     positive definite lattice of rank 2.
     Conversely it is known that for an even positive definite lattice of rank 2, $M$ there exists
     a unique singular K3 surface $X_{M}$ up to isomorphism such that $T(X_{M})$ is isomorphic to $M$ (\cite{IS}, Theorem 4).
     Every singular K3 surface has a model over a number field.\\

     In this section we assume $k$ is an algebraically closed field of characteristic $p>2$ and the cardinality of $k$ is equal to or less than the
     cardinality of $\C$. Let $W$ be the ring of Witt vectors of $k$ and $K$ be the fraction field of $W$. We fix an isomorphism
     $\bar{K} \simeq \C$. Let $X$ be a K3 surface of Picard number 20 over $k$. Then $X$ is an ordinary K3 surface, in particular $X$ is of finite height.

     \begin{lem}\label{uni}
     Assume $X$ is a K3 surface of Picard number 20 over $k$. Then $NS(X) \otimes \Z _{p}$ is an unimodular $\Z _{p}$-lattice of square discriminant.
     \end{lem}
     \begin{proof}
     Since the height of $X$ is 1, the flat cohomology
     $H^{2}_{fl}(X,\Z _{p}(1))$ is a free $\Z_{p}$-module of rank 20.
     In the exact sequence (\cite{I1}, p.629)
     \begin{center}
     $ 0 \to NS(X) \otimes \Z _{p} \to H^{2}_{fl}(X, \Z _{p}(1)) \to T_{p}(Br_{X}) \to 0$,
     \end{center}
     $T_{p}(Br_{X})$ is a free $\Z_{p}$-module and the rank of $NS(X) \otimes \Z_{p}$ is equal to the rank of $H^{2}_{fl}(X,\Z_{p}(1))$.
     Therefore $T_{p}(Br_{X})=0$ and $NS(X) \otimes \Z _{p} \simeq H^{2}(X, \Z _{p}(1))$. Because the cup product pairing of
     $H^{2}(X, \Z _{p}(1))$ is unimodular of square discriminant (\cite{Og2}, Remark 1.5), so is the pairing of $NS(X) \otimes \Z _{p}$.

     \end{proof}
     \begin{lem}\label{pri}
     Let $L$ be a local field of mixed characteristic $(0,p)$ and $k$ be an algebraic closed extension of the residue field of $L$.
     Let $X_{L}$ be a K3 surface defined over $L$.
     Assume $X_{L}$ has a good reduction over the residue field.
     Let $X$ be the the reduction of $X_{L}$ over $k$.
      Then the embedding
     $$NS(X_{L} \otimes \bar{L}) \otimes \Z_{l} \hookrightarrow NS(X \otimes k) \otimes \Z_{l}$$
     is primitive
     for any prime number $l \neq p$.

     \end{lem}
     \begin{proof}
     We may assume
         $NS(X_{L}) = NS(X_{L} \otimes \bar{L})$.
         For a prime $l \neq p$, the canonical embedding
     $$NS(X_{L}) \otimes \Z _{l} \hookrightarrow H^{2}(X_{L}(\C) , \Z _{l}) \simeq H^{2}_{\acute{e}t}(X_{L} \otimes \bar{L}, \Z_{l})
     = H^{2}_{\acute{e}t}(X \otimes k , \Z _{l})$$
      factors through
     $$ NS(X_{L}) \otimes \Z_{l} \hookrightarrow NS(X) \otimes \Z_{l} \hookrightarrow H^{2}_{\acute{e}t}(X \otimes k, \Z_{l}).$$
     Since the embedding $NS(X) \hookrightarrow H^{2}(X_{L}(\C), \Z)$ is primitive,
     $(NS(X)/NS(X_{L})) \otimes \Z_{l}$ has no torsion.
     \end{proof}
     \begin{rem}
     In the above lemma, the quotient group $NS(X) / NS(X _{L})$ may have a non-trivial $p$-torsion.
     It is known that if the ramification index of $L$ is less than $p-1$, $NS(X)/NS(X_{L})$ is torsion-free (\cite{Ra}, Th\'{e}or\`{e}me 4.1.2).
     \end{rem}
     Let us fix an embedding
      $$\bar{\Q} \hookrightarrow \C \simeq \bar{K}.$$ On each number field $F$, there exists a unique finite place of residue characteristic $p$ associated to the embedding $\bar{\Q} \hookrightarrow \bar{K}$.
        For any lattice $M$, we denote the discriminant of $M$ by $d(M)$.

     \begin{lem}\label{red}
     Let $F$ be a number field and $X_{F}$ be a singular K3 surface defined over $F$. Let $\upsilon$ be the place of $F$ associated to the
     embedding $\bar{\Q} \hookrightarrow \bar{K}$.
     Then $X_{F}$ has a potentially good reduction at $\upsilon$.
     Let $k$ be an algebraically closed extension of the residue field of $\upsilon$ and
     $X$ be a good reduction of $X_{F}$ over $k$.
     We suppose
       $d(NS(X_{F} \otimes \bar{F}))$ is a non-zero square modulo $p$.
     Then $X$ is ordinary and
     $NS(X) = NS(X _{F} \otimes \bar{F})$.

     \end{lem}
     \begin{proof}
     By \cite{Ma}, Corollary 0.5, $X$ has a potentially good reduction at $\upsilon$.
     Since there is an embedding
     $$NS(X_{F} \otimes \bar{F}) \hookrightarrow NS(X),$$
     by the assumption,
     $X$ is supersingular of Artin invariant 1 or ordinary. Since $d(NS(X_{F} \otimes \bar{F}))$
     is non-zero square modulo $p$, there is no embedding of $NS(X_{F} \otimes \bar{F}) \otimes \Z _{p}$ into
     the N\'{e}ron-Severi group of a supersingular K3 surface (\cite{Sh2}, Theorem 1.1). Hence $X$ is ordinary.
     Since $X$ is ordinary, the Picard number of $X$ is 20 and
           $NS(X)=NS(X _{F} \otimes \bar{F})$ by Lemma \ref{pri}.
     \end{proof}

     \begin{cor}[\cite{Li}, Theorem 2.6]\label{mode}
     Every K3 surface $X$ over $k$ of Picard number 20 has a model over a finite field.
     \end{cor}
     \begin{proof}
     Assume $\mathbb{X}$ is a N\'{e}ron-Severi group preserving lifting of $X$. Then $X_{\bar{K}} = \mathbb{X} \otimes \bar{K}$ is a singular K3 surface and
     has a model over a number field. We assume $X'_{F}$ is a K3 surface defined over a number field $F$ such that $X'_{F} \otimes \bar{K}$ is isomorphic to
     $X_{\bar{K}}$. Let $\upsilon$ be the place of $F$ corresponding to the embedding $\bar{\Q} \hookrightarrow \bar{K}$.
     We may assume $X'_{F}$ has good reduction $X'_{\upsilon}$ at $\upsilon$ by lemma \ref{red}.
     Then $X$ is isomorphic to $X' _{\upsilon} \otimes k$.
     \end{proof}

     \begin{prop}
     Let $X_{F}$ and $X'_{F}$ be two K3 surfaces of Picard number 20 defined over a number field $F$.
     Let $\upsilon$ be a place of $F$ whose residue characteristic is $p>2$ and
      $k$ be an algebraically closed extension of the residue field $k _{\upsilon}$.
      Assume $d(NS(X _{F} \otimes \bar{F}))$ and $d(NS(X'_{F} \otimes \bar{F}))$ are not divisible by $p$.
     If the reduction of $X_{F}$ over $k$ is isomorphic to the reduction of $X'_{F}$ over $k$ and
     both are ordinary,
     $X_{F} \otimes \bar{F}$ is isomorphic to $X'_{F} \otimes \bar{F}$.
     \end{prop}
     \begin{proof}
     It is enough to show that
     $X_{F} \otimes \bar{K}$ is isomorphic to $X' _{F} \otimes \bar{K}$.
     We set $L = K F_{\upsilon} \subset \bar{K}$ and $V$ is the ring of integers of $L$.
     Assume $\mathbb{X}$ and $\mathbb{X}'$ are smooth integral models of
     $X_{F}$ and $X_{F}'$ over $V$ respectively.
     Let $X$ be the common special fiber of $\mathbb{X}$ and $\mathbb{X}'$.
     By the assumption, $X$ is ordinary and
     $$NS(X) = NS(X_{F}) = NS(X'_{F}).$$
     Since the Picard number of $X$ is 20, the deformation space of $X$ together with $NS(X)$ is formally smooth of zero-dimensional over $W$
     (\cite{LM}, Proposition 4.1). Therefore $\mathbb{X}$ is isomorphic to $\mathbb{X}'$ and
     $X_{F} \otimes \bar{K}$ is isomorphic to $X'_{F} \otimes \bar{K}$.
     \end{proof}

     We summarize all the above results to give a classification of K3 surfaces of Picard number 20 over $k$.
     Let $S_{p}$ be the set of isomorphic classes of positive definite even $\Z$-lattices of rank 2 such that the discriminant
     is a non-zero square modulo $p$.
     For each $M \in S_{p}$, there is a unique singular K3 surface $X_{M}$ over $\bar{\Q}$ such that
     $T(X_{M})$ is isomorphic to $M$ via the given embedding $\bar{\Q} \hookrightarrow \bar{K} \simeq \C$. Let $X_{k,M}$ be the reduction of $X_{M}$
     over $k$ given in the proof of Lemma 3.4.
     The correspondence $M \mapsto X_{k,M}$ is a bijection from $S_{p}$ to the set of isomorphic classes of K3 surfaces of Picard number 20 over $k$.
     \begin{thm} Let $k$ be an algebraically closed field of characteristic $p>2$.
    The isomorphism classes of K3 surfaces of Picard number 20 are classified by $S_{p}$.
      Every K3 surface of Picard number 20 over $k$ has a model over a finite field.

     \end{thm}
     \begin{rem}

      In \cite{Li}, in a similar argument, C.Liedtke proves that every K3 surface of Picard number 20 over a field of odd characteristic has a model over a finite field and has a Shioda-Inose type ``sandwich''.
       \end{rem}

\vskip 1cm

\noindent
J.Jang\\
Department of Mathematics\\
University of Ulsan \\
Daehakro 93, Namgu Ulsan 680-749, Korea\\ \\
jmjang@ulsan.ac.kr


\begin{thebibliography}{99}
     %\bibitem{GR} Gros,M. Classes de Chern et classes de cycles en
     %cohomologie de Hodge-Witt logarithmique,Memoire SMF 21,1985
     \bibitem{A} Artin, M. Supersingular K3 surfaces, Ann. Sci. \'{E}cole Norm. Sup. (4) 7, 1974, 545--567.
     \bibitem{AM} Artin, M. and Mazur,B. Formal groups arising from algebraic varieties, Ann. Sci. \'{E}cole Norm. Sup(4) 10,1977,87--131
     %\bibitem{AS} Artin,M. and Swinnerton-Dyer,H.P.F. The Shafarecivh-Tate conjecture for pencils of elliptic curves on K3 surfaces,
     Invent. Math. 20, 1973, 249--266.
     %\bibitem{BO} Berthelot,P. and Ogus,A. A note on the crystalline cohomology
      \bibitem{C} Charles, F. The Tate conjecture for K3 surfaces over finite fields, to appear at Invent. Math.
     %\bibitem{CD} Cossec,F.R. and Dogachev,I.V. Enriques Surfaces I, Progress in Mathematics 76, 1989.
     \bibitem{D} Deligne, P.  Rel\`{e}vement des surfaces K3 en caract\'{e}ristique nulle, Surfaces Alg\'{e}briques,
     Lecture notes in Mathematics 868, 1981, 58--79.
     %\bibitem{Ho} Horikawa,E. On the periods of Enriques surfaces I,II, Math.Ann 234,1978,73-88, Math.Ann 235,1978,217--246.
     \bibitem{I1} Illusie, L. Complexe de de Rham-Witt et
     cohomologie cristalline, Ann. ens 12, 1979, 501--661.
     \bibitem{IT} Illusie, L. Grothendieck's existence theorem in formal geometry, Funeamental algebraic geometry,
     Mathematical surveys and monographs 123, 2005, 179--234.
     \bibitem{IS} Inose, H. and Shioda, T. On singular K3 surfaces, Complex analysis and algebraic geometry, 1977, Iwanami Shoten, 119--136.
     \bibitem{Ja} Jang, J. An Enriques involution of a supersingular K3 surface over odd characteristic, IMRN 2014, No.11, 3158--3175.
     %\bibitem{Ka} Katz,N. Slope filtraiton of F-crystal, Ast\'{e}risque 63, 1979, 113--163.
     \bibitem{Ke} Keum, J. Every algebraic Kummer surface is the K3-cover of an Enriques surface. Nagoya Math. J. 118, 1990, 99--110.
     %\bibitem{Ko} Kov\'{a}cs, J. The cone of curves of a K3 surface. Math. Ann. 300, 1994, 681--691.
     \bibitem{LM} Lieblich, M. and Maulik, D. A note on the cone conjecture for K3 surfaces in positive characteristic. matharxiv:1102.3377
       \bibitem{Li} Liedtke,C. Supersingular K3 surfaces are unirational. matharxiv:1304.5623.
  \bibitem{MA} Madapusi Pera, K. The Tate conjecture for K3 surfaces in odd characterisitc, matharxiv:1301.6326.
  \bibitem{MM} Matsusaka, T. and Mumford, D. Two fundamental theorems on deformations of Polarized varieties.
 American journal of mathematics 86, 1964, 668--684.
 \bibitem{Ma} Matsumoto, Y. On good reduction of some K3 surfaces related to abelian surfaces, matharxiv:1202.2421
     %\bibitem{Mo} Morrison,D.R. On K3 surfaces with large Picard number, Invent. Math 75, 1984, 105--121.
      \bibitem{MN} Mukai, S. and Namikawa,Y. Automorphisms of Enriques surfaces which acts trivially on the cohomology groups,
                   Invent. Math. 77, 1984, 383--397.
     \bibitem{Na} Namikawa, Y. Periods of Enriques surfaces, Math.Ann 270,1985,201--222.
     \bibitem{Ni} Nikulin, V.V. Kummer surface, Izv. akad. nauk SSSR 39, 1975, 278--293.
     %\bibitem{N1} Nyhaard,N.O. A p-adic proof of the nonexistence of vectore fields on K3 surface, Ann. of Math. 110, 1979, 515--528.
     \bibitem{N2} Nygaard, N.O. The Tate conjecture for ordinary K3 surfaces over finite fields, Invent. Math 74, 1983, 213--237.
     \bibitem{NO} Nygaard, N.O. and Ogus,A. Tate conjecture for K3 surfaces of finite height, Ann. of Math. 122, 1985, 461--507.
     \bibitem{Og1} Ogus, A. Supersingular K3 crystal, ast\'{e}risque 64, 1979, 3--86
     \bibitem{Og2} Ogus, A. A crystalline Torelli theorem for supersingular K3 surfaces, Prog. Math. 36, 1983, 361-394.
     \bibitem{PS} Piatetski-Shapiro, I and Shafarevich, I.R. A Torelli theorem for algebraic surfaces of type K3, Math. USSR Izv. 5, 1971, 547--588.
      \bibitem{Ra} Raynaud, M. p-torsion du sch\'{e}ma de Picard, ast\'{e}risque 64, 1979, 87--148.
     %\bibitem{RS} Rudakov,A.N. and Shafarevich,I.R. Inseparable morphism of algebraic surfaces, Izv.Akad.Nauk SSSR Ser.Mat. 40, 1976, 1269--1307.
     %\bibitem{Sh1} Shimada,I. Supersingular K3 surfaces in odd characteristic and sextic double planes, Math.Ann. 328, 2004, 451--468.
     \bibitem{Sh2} Shimada, I. On normal K3 surfaces, Michigan Math. J. 55, 2007, 395--416
     %\bibitem{Sh3} Shimada,I. Transcendental Lattices and supersingular reduction lattices of a singular K3 surface, Trans. of A.M.S. 361, 2009, 909--949.

     \end{thebibliography}
     \end{document}